\documentclass{article}
\usepackage{amsmath, amsthm}
\usepackage{amssymb}
\newtheoremstyle{theorem}
  {10pt}          
  {10pt}  
  {\sl}  
  {\parindent}     
  {\bf}  
  {. }    
  { }    
  {}     
\newtheorem{thm}{Theorem}[section]

\newtheorem{prop}[thm]{Proposition}
\newtheorem{defn}{Definition}[section]

\begin{document}
\title{\large\bf Nabla Euler -Lagrange equations in discrete fractional variational calculus within Riemann and Caputo}
\author{\small \bf Thabet Abdeljawad $^a$  \\  {\footnotesize $^a$ Department of Mathematics and Physical Sciences}\\
{\footnotesize Prince Sultan University, P. O. Box 66833, Riyadh 11586, Saudi Arabia}}
\date{}
\maketitle

{\footnotesize {\noindent\bf Abstract.} Different fractional difference types of Euler-Lagrange equations are obtained within Riemann and Caputo by making use of different versions of integration by part forumlas  in fractional difference calculus. An example is presented to illustrate part of the results.
 \\


\section{Introduction}
Fractional calculus which deals with integration and differentiation of arbitrary orders attracted the attention of many researchers in the last two decades or so for its widespread applications in different fields of mathematics, physics, engineering, economic and biology. For  detailed and sufficient material about this calculus we refer to the books \cite{podlubny, Samko, Kilbas}. However, the discrete fractional calculus which is not as old as fractional calculus,  was initiated lately in eighty's of the last century in \cite{Miller, Gray}. Then in the last few years many authors started to investigate the theory and applications of the discrete fractional calculus \cite{Th Caputo,Ferd, Feri, Nabla, Atmodel, Gronwall, TDbyparts, Gdelta, Gnabla, Gfound}. Very recently, the authors in \cite{THFer, Th DDNS, Th ADE} have discussed different definitions for fractional differences specially in the right case, under which suitable integration by parts formulae have been initiated. Benefitting from those formulae we continue in this work and apply to discrete fractional variational calculus to obtain different results from those obtained in \cite{Nuno, Atmodel}. In the usual fractional variational case we refer to \cite{FTD1, FTD2, FTD3,tr, FTD4}.

The article is organized as follows: In the rest of this section we give basic definitions and preliminary results about nabla fractional sums and differences. In Section 2 we discussed different integration by parts formulae in discrete  fractional calculus. In Section 3 we set some discrete variational problems benefitting from the integration by parts formulae obtained in Section 2. Finally, in Section 4 an example of physical interest is given to illustrate our main results.

For the sake of the nabla fractional calculus we have the following definition

\begin{defn} \label{rising}(\cite{Adv,Boros,Grah,Spanier})

(i) For a natural number $m$, the $m$ rising (ascending) factorial of $t$ is defined by

\begin{equation}\label{rising 1}
    t^{\overline{m}}= \prod_{k=0}^{m-1}(t+k),~~~t^{\overline{0}}=1.
\end{equation}

(ii) For any real number the $\alpha$ rising function is defined by
\begin{equation}\label{alpharising}
 t^{\overline{\alpha}}=\frac{\Gamma(t+\alpha)}{\Gamma(t)},~~~t \in \mathbb{R}-~\{...,-2,-1,0\},~~0^{\mathbb{\alpha}}=0
\end{equation}

\end{defn}

Regarding the rising factorial function we observe for example that

 \begin{equation}\label{oper}
    \nabla (t^{\overline{\alpha}})=\alpha t^{\overline{\alpha-1}}
\end{equation}

\textbf{Notation}:
\begin{enumerate}
\item[$(i)$] For a real $\alpha>0$, we set $n=[\alpha]+1$, where $[\alpha]$ is the greatest integer less than or equal to $\alpha$.

\item[$(ii)$] For real numbers $a$ and $b$, we denote $\mathbb{N}_a=\{a,a+1,...\}$ and $~_{b}\mathbb{N}=\{b,b-1,...\}$.

\item[$(iii)$]For $n \in \mathbb{N}$ and real $a$, we denote
$$ _{\circleddash}\Delta^n f(t)\triangleq (-1)^n\Delta^n f(t),$$ where $\Delta^n f$ is the $n$ iterating of $\Delta f(t)=f(t+1)-f(t)$.
\end{enumerate}

\begin{defn} \label{fractional sums} \cite{Th ADE}
Let $\sigma(t)=t+1$ and $\rho(t)=t-1$ be the forward and backward jumping operators, respectively. Then

(i) The (nabla) left fractional sum of order $\alpha>0$ (starting from $a$) is defined by:
\begin{equation}\label{nlf}
  \nabla_a^{-\alpha} f(t)=\frac{1}{\Gamma(\alpha)} \sum_{s=a+1}^t(t-\rho(s))^{\overline{\alpha-1}}f(s),~~t \in \mathbb{N}_{a+1}.
\end{equation}

(ii)The (nabla) right fractional sum of order $\alpha>0$ (ending at $b$) is defined by:
\begin{equation}\label{nrs}
   ~_{b}\nabla^{-\alpha} f(t)=\frac{1}{\Gamma(\alpha)} \sum_{s=t}^{b-1}(s-\rho(t))^{\overline{\alpha-1}}f(s)=\frac{1}{\Gamma(\alpha)} \sum_{s=t}^{b-1}(\sigma(s)-t)^{\overline{\alpha-1}}f(s),~~t \in ~_{b-1}\mathbb{N}.
\end{equation}
\end{defn}

\begin{defn} \label{fractional differences} \cite{Th ADE}

(i) The (nabla) left fractional difference of order $\alpha>0$ (starting from $a$ ) is defined by:
\begin{equation}\label{nld}
  \nabla_a^{\alpha} f(t)=\nabla^n \nabla_a^{-(n-\alpha)}f(t)= \frac{\nabla^n}{\Gamma(n-\alpha)} \sum_{s=a+1}^t(t-\rho(s))^{\overline{n-\alpha-1}}f(s),~~t \in \mathbb{N}_{a+1}
\end{equation}

(ii) The (nabla) right fractional difference of order $\alpha>0$ (ending at $b$ ) is defined by:
\begin{equation}\label{nrd}
   ~_{b}\nabla^{\alpha} f(t)= ~_{\circleddash}\Delta^n ~_{b}\nabla^{-(n-\alpha)}f(t) =\frac{(-1)^n\Delta^n}{\Gamma(n-\alpha)} \sum_{s=t}^{b-1}(s-\rho(t))^{\overline{n-\alpha-1}}f(s),~~t \in ~ _{b-1}\mathbb{N}
\end{equation}

\end{defn}

\begin{defn} \cite{Th DDNS}
Let $\alpha >0$ be noninteger, $~ n=[\alpha]+1,~a(\alpha)=a+n-1$ and $b(\alpha)=b-n+1$. Then the (dual) nabla left and right Caputo fractional differences are defined by
\begin{equation}\label{Cdual left}
   ~^{C}\nabla_{a(\alpha)}^\alpha f(t)=\nabla_{a(\alpha)}^{-(n-\alpha)} \nabla^n f(t),~~t \in \mathbb{N}_{a+n}
\end{equation}

and
\begin{equation}\label{Cdual right}
  _{b(\alpha)}~ ^{C}\nabla^\alpha f(t)=~_{b(\alpha)}\nabla^{-(n-\alpha)} {\ominus}\Delta^n f(t), ~~t \in ~_{b-n}\mathbb{N},
\end{equation}
respectively.
\end{defn}
Notice that  when $0<\alpha < 1$ we have
$$~^{C}\nabla_{a(\alpha)}^\alpha f(t)=~^{C}\nabla_a^\alpha f(t)~~and ~~ _{b(\alpha)} ^{C}\nabla^\alpha f(t)= ~_{b} ^{C}\nabla^\alpha f(t).$$

It is important to remark that the two quantities $(\nabla_a^{-\alpha}f^\rho)(t)$ and $(\nabla_a^{-\alpha}f)(\rho(t))$ are different, where $\rho(t)=t-1$. In connection, we state the following properties without proofs.

\begin{prop}\label{properties}
Let $\rho(t)=t-1$, $\sigma(t)=t+1$ ,  $\alpha >0$ and $f$ be function defined on $\mathbb{N}_a \cap ~_{b}\mathbb{N}$ where $a\equiv b ~(mod ~1)$. Then
\begin{itemize}
  \item 1) $(\nabla_a^ {-\alpha}f^\rho)(t)=(\nabla_{a-1}^ {-\alpha}f)(\rho(t))$.
  \item 2) $(\nabla_a^ {\alpha}f^\rho)(t)=(\nabla_{a-1}^ {\alpha}f)(\rho(t))$.
  \item 3) $(~^{C}\nabla_a^\alpha f^\rho)(t)=(~^{C}\nabla_{a-1}^\alpha f)(\rho(t))$.
  \item 4)$(~_{b}\nabla^{-\alpha}f^\sigma)(t)=(~_{b+1}\nabla^{-\alpha}f)(\sigma(t))$.
  \item 5)$(~_{b}\nabla^{\alpha}f^\sigma)(t)=(~_{b+1}\nabla^{\alpha}f)(\sigma(t))$.
  \item 6)$(~^{C}_{b}\nabla^\alpha f^\sigma)(t)=(~^{C}_{b+1}\nabla^\alpha f)(\sigma(t))$.
   \end{itemize}
\end{prop}
\section{Integration by parts for fractional sums and differences}

In this section we state the integration by parts formulas for nabla fractional sums and differences obtained in \cite{THFer}, whereafter in \cite{Th ADE}, delta by parts formulas are obtained by using certain dual identities. Then, we proceed to obtain  a one more integration by parts formula where both Riemann and Caputo fractional differences can appear.

\begin{prop} \label{summation by parts}\cite{THFer}
For $\alpha>0$, $a,b \in \mathbb{R}$, $f$ defined on $\mathbb{N}_a$ and $g$ defined on $~_{b}\mathbb{N}$, we have

\begin{equation}\label{sum1}
    \sum_{s=a+1}^{b-1}g(s) \nabla_a^{-\alpha} f(s)=\sum_{s=a+1}^{b-1}f(s)~_{b}\nabla^{-\alpha}g(s).
\end{equation}
\end{prop}

\begin{prop} \label{nabla bydiff} \cite{THFer}
Let $\alpha>0$ be non-integer and $a,b\in \mathbb{R}$ such that $a< b$ and $b\equiv
a~(mod~1)$.If $f$ is defined on $ _{b}\mathbb{N}$ and $g$ is
defined on $\mathbb{N}_a$, then
\begin{equation}\label{nabla bydiff1}
\sum_{s=a+1}^{b-1} f(s)\nabla_a^\alpha g(s)
=\sum_{s=a+1}^{b-1}g(s)~_{b}\nabla^\alpha f(s).
\end{equation}
\end{prop}

Now by the above nabla integration by parts formulas and using  dual identities in \cite{Th ADE}, the following  delta integration by parts formulae were obtained.

\begin{prop} \label{delta by parts semmation}

Let $\alpha>0$, $a,b\in \mathbb{R}$ such that $a< b$ and $b\equiv
a~(mod~1)$. If $f$ is defined on $\mathbb{N}_a$ and $g$ is defined on
$_{b}\mathbb{N}$, then we have
\begin{equation}\label{byse}
\sum_{s=a+1}^{b -1}g(s)(\Delta_{a+1}
^{-\alpha}f)(s+\alpha)=\sum_{s=a+1}^{b-1}f(s) ~_{b-1}\Delta ^{-\alpha}g(s-\alpha).
\end{equation}

\end{prop}

\begin{prop} \label{delta by parts semmation}

Let $\alpha>0$ be non-integer and assume that $b\equiv a~(mod~1)$. If $f$ is defined on $ _{b}\mathbb{N}$ and $g$ is
defined on $\mathbb{N}_a$, then
\begin{equation}\label{bydiff1}
\sum_{s=a+1}^{b-1} f(s)\Delta_{a+1}^\alpha
g(s-\alpha)=\sum_{s=a+1}^{b-1}g(s)~_{b-1}\Delta^\alpha f(s+\alpha).
\end{equation}

\end{prop}
The following version of integration by parts contains boundary conditions.

\begin{thm} \label{Caputo by parts} \cite{Th DDNS}
Let $0<\alpha<1$ and $f,g$ be functions defined on $\mathbb{N}_a \cap ~_{b}\mathbb{N}$ where $a\equiv b ~(mod ~1)$. Then
\begin{equation} \label{cbp1}
\sum_{s=a+1}^{b-1} g(s) ~^{C}\nabla_a^\alpha f(s)=f(s) ~_{b}\nabla^{-(1-\alpha)}g(s)\mid_a^{b-1}+ \sum_{s=a+1}^{b-1} f(s-1) (~_{b}\nabla^\alpha g)(s-1),
\end{equation}
where clearly $_{b}\nabla^{-(1-\alpha)}g(b-1)= g(b-1)$.
\end{thm}
Similarly, if interchange the role of Caputo and Riemann we obtain the following version of integration by parts for fractional differences.
\begin{thm} \label{Riemann2 by parts}
Let $0<\alpha<1$ and $f,g$ be functions defined on $\mathbb{N}_a \cap ~_{b}\mathbb{N}$ where $a\equiv b ~(mod ~1)$. Then

\begin{eqnarray}\label{cbp1}\nonumber
  \sum_{s=a+1}^{b-1} f(s-1) \nabla_a^\alpha g(s) &=& f(s) \nabla_a^{-(1-\alpha)}g(s)\mid_a^{b-1}+ \sum_{s=a}^{b-2} g(s+1) ~~^{C}_{b}\nabla^\alpha f(s) \\
   &=& f(s) \nabla_a^{-(1-\alpha)}g(s)\mid_a^{b-1}+ \sum_{s=a+1}^{b-1} g(s) ~~(^{C}_{b}\nabla^\alpha f)(s-1)
\end{eqnarray}

where clearly $\nabla_a^{-(1-\alpha)}g(a)= 0$.
\end{thm}
\begin{proof}
From the definition of the left Riemann fractional difference, the integration by parts in $\nabla-$difference calculus,  Proposition \ref{summation by parts}, noting that $\nabla f(s)=\Delta f(s-1)$, and the definition of right Caputo fractional difference we can write
\begin{eqnarray} \nonumber
  \sum_{s=a+1}^{b-1} f(s-1) \nabla_a^\alpha g(s) &=& \sum_{s=a+1}^{b-1} f(s-1)\nabla \nabla_a^{-(1-\alpha)} g(s) \\ \nonumber
   &=& f(s) \nabla_a^{-(1-\alpha)}g(s)\mid_a^{b-1}-\sum_{s=a+1}^{b-1} \nabla f(s) (\nabla_a^{-(1-\alpha)}g)(s) \\ \nonumber
   &=& f(s) \nabla_a^{-(1-\alpha)}g(s)\mid_a^{b-1}-\sum_{s=a+1}^{b-1}g(s)~_{b}\nabla^{-(1-\alpha)} \nabla f(s)\\ \nonumber
   &=&f(s) \nabla_a^{-(1-\alpha)}g(s)\mid_a^{b-1}+ \sum_{s=a}^{b-2} g(s+1) ~~^{C}_{b}\nabla^\alpha f(s)\\ \nonumber
   &=&f(s) \nabla_a^{-(1-\alpha)}g(s)\mid_a^{b-1}+ \sum_{s=a+1}^{b-1} g(s) ~~(^{C}_{b}\nabla^\alpha f)(s-1).
    \end{eqnarray}

 Hence, the proof is completed.

\end{proof}
\section{Fractional difference Euler-Lagrange Equations}

\begin{thm}\label{m1}
Let $\alpha>0$ be non-integer, $a,b \in \mathbb{R}$, and  $f$ is defined on $\mathbb{N}_a \cap ~_{b}\mathbb{N}$, where $a\equiv b ~(mod ~1)$. Assume that the functional
 $$J(f)=\sum_{t=a+1}^{b-1} L(t,f(t),\nabla_{a-1}^\alpha f(t) )$$
 has a local extremum in $S=\{y:\mathbb{N}_a \cap ~_{b}\mathbb{N}\rightarrow \mathbb{R}~\texttt{is bounded}, ~~y(a)=A\}$ at some $f \in S$, where
 $L:(\mathbb{N}_a \cap ~_{b}\mathbb{N})\times \mathbb{R}\times \mathbb{R}\rightarrow \mathbb{R}$. Then,
\begin{equation}\label{E1}
[L_1(s) + ~_b\nabla^\alpha L_2(s)]  =0,~\texttt{for all}~ s \in (\mathbb{N}_{a+1} \cap ~_{b-1}\mathbb{N}),
\end{equation}
where $L_1(s)= \frac{\partial L}{\partial f}(s)$ and $L_2(s)=\frac{\partial L}{\partial \nabla_a^\alpha f}(s)$.
\end{thm}
\begin{proof}
Without loss of generality, assume that $J$ has local maximum in $S$ at $f$. Hence, there exists an $\epsilon>0$ such that $J(\widehat{f})-J(f)\leq 0$ for all $\widehat{f}\in S$ with $\|\widehat{f}-f\|=\sup_{t \in \mathbb{N}_a \cap ~_{b}\mathbb{N}} |\widehat{f}(t)-f(t)|< \epsilon$. For any $\widehat{f} \in S$ there is an $\eta \in H=\{y:\mathbb{N}_a \cap ~_{b}\mathbb{N}\rightarrow \mathbb{R}~\texttt{is bounded}, ~~y(a)=0\}$ such that $\widehat{f}=f+\epsilon \eta$. Then, the $\epsilon-$Taylor's theorem implies that
$$L(t,f,\widehat{f})=L(t,f+\epsilon \eta,\nabla_{a-1}^\alpha f+\epsilon \nabla_{a-1}^\alpha \eta)=L(t,f,\nabla_{a-1}^\alpha f)+ \epsilon [\eta L_1+\nabla_a^\alpha\eta L_2]+O(\epsilon^2).$$ Then,

\begin{eqnarray}\nonumber
  J(\widehat{f})-J(f) &=& \sum_{t=a+1}^{b-1}L(t,\widehat{f}(t),\nabla_{a-1}^\alpha \widehat{f}(t))-\sum_{t=a+1}^{b-1}L(t,f(t),\nabla_{a-1}^\alpha f(t)) \\
  &=& \epsilon \sum_{t=a+1}^{b-1}[\eta(t) L_1(t)+ \nabla_{a-1}^\alpha \eta(t) L_2(t)]+ O(\epsilon^2).
   \end{eqnarray}
Let the quantity $\delta J(\eta,y)=\sum_{t=a+1}^{b-1}[\eta(t) L_1(t)+ \nabla_{a-1}^\alpha \eta(t) L_2(t)]$ denote the first variation of $J$.

Evidently, if $\eta \in H$ then $-\eta \in H$, and $\delta J(\eta,y)=-\delta J(-\eta,y)$. For $\epsilon$ small, the sign of $J(\widehat{f})-J(f)$ is determined by the sign of first variation, unless $\delta J(\eta,y)=0$ for all $\eta \in H$. To make the parameter $\eta$ free, we use the integration by part formula in Proposition \ref{nabla bydiff} together with the fact that $\nabla_{a-1}^\alpha \eta(t)=\nabla_{a}^\alpha \eta(t)+\frac{\eta(a)}{\Gamma(\alpha)}(t-a+1)^{\overline{\alpha-1}}$, to reach
$$\delta J(\eta,y)=\sum_{s=a+1}^{b-1} \eta (s)[L_1(s) + ~_b\nabla^\alpha L_2(s)]++\frac{\eta(a)}{\Gamma(\alpha)}\sum_{t=a+1}^{b-1}(t-a+1)^{\overline{\alpha-1}}L_2(t) =0,$$ for all $\eta \in H$, and hence the result follows by taking the special $\eta's$ in $\{e_t=(0,...,1,0,0,), ~1~\texttt{in t-th place}: t \in \mathbb{N}_{a+1} \cap ~_{b-1}\mathbb{N} \}$ with $\eta(a)=0$.
\end{proof}
Note that in the above theorem  the Riemann fractional variational difference problem will not require any boundary conditions at the points $a$ and $b-1$ if we consider $\nabla_a^\alpha$ instead of $\nabla_{a-1}^\alpha$ in the Lagrangian $L$ and hence the functions $\eta$ can be taken from $S$ again without any restrictions. This is due to that the used integration by parts formula does not contain any boundary conditions. Different boundary conditions can be generated at $b$ as well, if we terminate the sum at $b$ instead of $b-1$ . Next, we develop a discrete Reiamnn fractional variational problem of order $0<\alpha <1$ with different boundary conditions by making use of the integration by part formula in Theorem \ref{Riemann2 by parts}.

\begin{thm} \label{mm}
Let $0< \alpha <1$ be non-integer, $a,b \in \mathbb{R}$, and  $f$ is defined on $\mathbb{N}_a \cap ~_{b}\mathbb{N}$, where $a\equiv b ~(mod ~1)$. Assume that the functional
 $$J(y)=\sum_{t=a+1}^{b-1} L(t,f(t),\nabla_a^\alpha f(t) )$$
 has a local extremum in $S=\{y:\mathbb{N}_a \cap ~_{b}\mathbb{N}\rightarrow \mathbb{R}~\texttt{is bounded}\}$ at some $f \in S$, where
 $L:(\mathbb{N}_a \cap ~_{b}\mathbb{N})\times \mathbb{R}\times \mathbb{R}\rightarrow \mathbb{R}$. Further, assume either $\nabla_a^{-(1-\alpha)}f (b-1)=A$ or $L_2 (b)=0$. Then,
\begin{equation}\label{E2}
[L_1(s) + (~^{C}_b\nabla^\alpha L_2^\sigma)(s-1)]=[L_1(s) + (~^{C}_{b+1}\nabla^\alpha L_2)(s)] =0,~\texttt{for all}~ s \in (\mathbb{N}_{a+1} \cap ~_{b-1}\mathbb{N}),
\end{equation}
where $L_1(s)= \frac{\partial L}{\partial f}(s)$ and $L_2(s)=\frac{\partial L}{\partial \nabla_a^\alpha f}(s)$.
\end{thm}
\begin{proof}
We proceed as in Theorem \ref{m1}, except when $\nabla_a^{-(1-\alpha)}f(b-1)$ is preassigned the function $\eta$ is taken from $H=\{y:\mathbb{N}_a \cap ~_{b}\mathbb{N}\rightarrow \mathbb{R}~\texttt{is bounded},~\nabla_a^{-(1-\alpha)}y(b-1)=0\}$. Then, $$\delta J(\eta,f)=\sum_{t=a+1}^{b-1}[\eta(t) L_1(t)+ \nabla_a^\alpha \eta(t) L_2^\sigma(t-1)]=0,$$ for every $\eta \in H$. Then, the integration by part in Theorem \ref{Riemann2 by parts}  implies that
$$\delta J(\eta,f)=\sum_{t=a+1}^{b-1}\eta(t)[ L_1(t)+ ~_{b}^{C}\nabla^\alpha  L_2^\sigma(t-1)]+ L_2^\sigma(t) \nabla_a ^{-(1-\alpha)}\eta(t)|_a^{b-1}=0,$$ for every $\eta \in H$. Finally, the assumption and Proposition \ref{properties} 6) implies (\ref{E2}) and the proof is finished.
\end{proof}

Similar to what applied in Theorem \ref{m1}, we can generate boundary conditions  at $a$ as well in Theorem \ref{mm} above, if we consider $\nabla_{a-1}^\alpha$ instead of $\nabla_{a}^\alpha$ in the Lagrangian $L$.

Finally, we obtain the Euler-Lagrange equations for a Lagrangian including the Caputo left fractional difference by making use of the integration by parts formula in Theorem \ref{Caputo by parts}.

\begin{thm} \label{mmm}
Let $0< \alpha <1$ be non-integer, $a,b \in \mathbb{R}$, and  $f$ are defined on $\mathbb{N}_a \cap ~_{b}\mathbb{N}$, where $a\equiv b ~(mod ~1)$. Assume that the functional
 $$J(f)=\sum_{t=a+1}^{b-1} L(t,f^\rho(t),~^{C}\nabla_a^\alpha f(t) )$$
 has a local extremum in $S=\{y:\mathbb{N}_a \cap ~_{b}\mathbb{N}\rightarrow \mathbb{R}~\texttt{is bounded}\}$ at some $f \in S$, where
 $L:(\mathbb{N}_a \cap ~_{b}\mathbb{N})\times \mathbb{R}\times \mathbb{R}\rightarrow \mathbb{R}$. Further, assume either $f(a)=A$ and $f(b-1)=B$ or the natural boundary conditions $~_{b}\nabla^{-(1-\alpha)}L_2 (a)=~_{b}\nabla^{-(1-\alpha)}L_2 (b-1)=0$. Then,
\begin{equation}\label{E3}
[L_1^\sigma(s) + (~_b\nabla^\alpha L_2)(s)] =0,~\texttt{for all}~ s \in (\mathbb{N}_{a+1} \cap ~_{b-2}\mathbb{N}).
\end{equation}

\end{thm}
\begin{proof}
If $f$ is preassigned at $a$ and $b-1$ then the function $\eta$ is taken from $H=\{y:\mathbb{N}_a \cap ~_{b}\mathbb{N}\rightarrow \mathbb{R}~\texttt{is bounded},~y(a)=y(b-1)=0\}$. Then, we proceed to reach $$\delta J(\eta,f)=\sum_{t=a+1}^{b-1}[\eta(t-1) L_1(t)+ ~^{C}\nabla_a^\alpha \eta(t) L_2(t)]=0,$$ for every $\eta \in H$. The integration by parts formula in Theorem \ref{Caputo by parts} then implies that
$$\delta J(\eta,f)=\sum_{t=a+1}^{b-1}\eta(t-1)[ L_1(t)+ (~_{b}\nabla^\alpha  L_2)(t-1)]+\eta(t) ~_{b}\nabla^{-(1-\alpha)}L_2(t)|_a^{b-1}=0,$$ for every $\eta \in H$. Hence, (\ref{E3}) follows.
\end{proof}

We finish this section by remarking that we can obtain a delta analogue of the discussed nabla discrete variational problems in this section by making use of the dual identities studied in \cite{Th DDNS, Th ADE}.
\section{Example}
In order to exemplify our results we analyze an example of physical interest under Theorem \ref{mm} and Theorem \ref{mmm}. Namely, let us consider the following fractional discrete actions,
\begin{itemize}
  \item 1) $J(y)=\sum_{t=a+1}^{b-1}[\frac{1}{2}(\nabla_a^\alpha y(t))^2-V(y(t))],$ where $0<\alpha <1$. Assume either $\nabla_a^{-(1-\alpha)}f (b-1)=A$ or  $\nabla_a^\alpha y(b)=0$. Then  the Euler-Lagrange equation by applying Theorem \ref{mm} is
       $$~^{C}_{b+1}\nabla^\alpha \nabla_a^\alpha y(s)-\frac{dV}{dy}(s)=0~\texttt{for all}~ s \in (\mathbb{N}_{a+1} \cap ~_{b-1}\mathbb{N}).$$
  \item 2)$J(y)=\sum_{t=a+1}^{b-1}[\frac{1}{2}(~^{C}\nabla_a^\alpha y(t))^2-V(y(\rho(t)))],$ where $0<\alpha <1$.  Assume either $y(a)=A$ and $y(b-1)=B$ or the natural boundary conditions $~_{b}\nabla^{-(1-\alpha)}~^{C}\nabla_a^\alpha (a)=~_{b}\nabla^{-(1-\alpha)} ~^{C}\nabla_a^\alpha y(b-1)=0$. Then  the Euler-Lagrange equation by applying Theorem \ref{mmm} is
       $$ (~_b\nabla^\alpha ~^{C}\nabla_a^\alpha y)(s)-\frac{dV}{d y^\rho}(\sigma(s))] =0,~\texttt{for all}~ s \in (\mathbb{N}_{a+1} \cap ~_{b-2}\mathbb{N}).$$
Finally, we remark that it is of interest  to deal with the above Euler- Lagrange equations obtained in the above example, where we have composition of left and right fractional differences. In the usual fractional case for such left-right fractional dynamical systems we mention the work done in \cite{Thabet}.
\end{itemize}
\section{Acknowledgments}
The author would like to thank Prince Salman Research and  Translation Center in Prince Sultan University for the financial support.


\begin{thebibliography}{99}

\bibitem{Th Caputo} T. Abdeljawad , On Riemann and Caputo fractional differences,
Computers and Mathematics with Applications, Volume 62, Issue 3, August 2011, Pages 1602-1611.


\bibitem{Thabet}T. Abdeljawad (Maraaba),  Baleanu D. and  Jarad F.,
Existence and uniqueness theorem for a class of delay differential
equations with left and right Caputo fractional derivatives, Journal of
Mathematical Physics, 49 (2008), 083507.

\bibitem{Ferd} F.M. At{\i}c{\i}  and Eloe P. W.,  A Transform method in
discrete fractional calculus, \emph{International Journal of
Difference Equations}, vol 2, no 2, (2007), 165--176.

\bibitem{Feri}F.M.  At{\i}c{\i}  and Eloe P. W.,  Initial value problems in
discrete fractional calculus, \emph{Proceedings of the American
Mathematical Society}, 137, (2009), 981-989.

\bibitem{Nabla} F. M. At{\i}c{\i} and Paul W.Eloe, Discrete fractional
calculus with the nabla operator, Electronic Journal of Qualitative Theory
of Differential equations, Spec. Ed. I, 2009 No.3,1--12.

\bibitem{Atmodel}F.M. At{\i}c{\i},  \c{S}eng\"{u}l S., Modelling with fractional difference equations, Journal of Mathematical Analysis and Applications,
369 (2010) 1-9.

\bibitem{Gronwall} F. M. At{\i}c{\i}, Paul W. Eloe, Gronwall's inequality on discrete fractional calculus, Copmuterand Mathematics with pplications, In Press, doi:10.1016/camwa. 2011.11.029.

\bibitem{Miller} K. S. Miller,  Ross B.,Fractional difference
calculus, \emph{Proceedings of the International Symposium on
Univalent Functions, Fractional Calculus and Their Applications}, Nihon University, Koriyama, Japan, (1989), 139-152.


\bibitem{podlubny}
I. Podlubny, Fractional Differential Equations, Academic Press: San
Diego CA, (1999).



\bibitem{Samko}Samko G. Kilbas A. A., Marichev, Fractional Integrals
and Derivatives: Theory and Applications, Gordon and Breach, Yverdon, 1993.

\bibitem{Kilbas} Kilbas A., Srivastava M. H.,and Trujillo J. J., Theory and Application of Fractional Differential Equations, North
Holland Mathematics Studies 204, 2006.

\bibitem{TDbyparts}T. Abdeljawad , D. Baleanu , Fractional Differences and
integration by parts, Journal of Computational Analysis and Applications
vol 13 no. 3 , 574-582 (2011).













\bibitem{FTD1}
Dumitru Baleanu, Thabet Abdeljawad,(Maraaba), Fahd Jarad, Fractional Variational principles with delay; Journal of Physics A : Mathematical and Theoretical, vol 41 issue 31, 315403 (2008).












\bibitem{Nuno} Nuno R. O. Bastos, Rui A. C. Ferreira, Delfim F. M.
Torres, Discrete-time fractional variational problems, Signal Processing,  91(3): 513-524 (2011).

\bibitem{Adv}Bohner M.  and A. Peterson, Advances in Dynamic Equations on Time Scales, Birkh\"{a}user, Boston, 2003.

\bibitem{Boros}G. Boros and V. Moll, Iresistible Integrals; Symbols,Analysis and Expreiments in the Evaluation of Integrals,
Cambridge University PressCambridge 2004.

\bibitem{Grah} R. L. Graham, D. E. Knuth and O. Patashnik, Concrete Mathematics, A Foundation for Copmuter Science,
2nd ed. , Adison-Wesley,Reading, MA, 1994.


\bibitem{Gray} H. L. Gray and N. F.Zhang, On a new definition of the fractional difference, Mathematics of Computaion 50, (182), 513-529 (1988.)

\bibitem{Spanier} J. Spanier and K. B. Oldham, The Pochhammer Polynomials $(x)_n$, An Atlas of Functions,
pp. 149-156,Hemisphere, Washington, DC, 1987.


\bibitem{Gdelta}G. A. Anastassiou,Principles of delta fractional calculus on time scales and inequalities, Mathematical and Computer Modelling, 52 (2010)556-566.
\bibitem{Gnabla} G. A. Anastassiou, Nabla discrete calcilus and nabla inequalities, Mathematical and Computer Modelling, 51 (2010) 562-571.

\bibitem{Gfound} G. A. Anastassiou,Foundations of nabla fractional calculus on time scales and inequalities,Computer and Mathematics with Applications, 59 (2010) 3750-3762.
\bibitem{FTD2} Fahd Jarad, Thabet Abdeljawad and Dumitru Baleanu, Fractional Variational Principles with Delay within Caputo Derivatives, Reports of Mathematical Physics, vo 65 , 17-28 (2010).


\bibitem{FTD3} Fahd Jarad, Thabet Abdeljawad,  Dumitru Baleanu, Fractional Variational Optimal Control Problems with Delayed Arguments, Nonlinear Dynamics , 62 (3), 609-614 (2010).



\bibitem{tr} D. Baleanu, J. J. Trujillo, On exact solutions of a class
o f fractional Euler-Lagrange equations, Nonlin.Dyn. 52(4) 331-335 (2008).
 \bibitem{FTD4}
Fahd Jarad, Thabet Abdeljawad,  Dumitru Baleanu, Higher order variational optimal control problems with delayed arguments, Applied Mathematics and Computation, 218 (18), 9234-9240 (2012).
\bibitem{THFer}T. Abdeljawad, F. At{\i}c{\i}, On the Definitions of Nabla Fractional Operators, Abstract and applied Analysis, 2012 (2012), Article ID 406757, 13 pages, doi:10.1155/2012/406757.

 \bibitem{Th DDNS}T. Abdeljawad, On Delta and Nabla Caputo fractional differences and dual identities, arXiv:1102.1625 (2013).

  \bibitem{Th ADE}T. Abdeljawad, Dual identities in fractional difference calculus within Riemann, to appear in Advances in Difference Equations, 2013, arXiv:1112.5795.

\end{thebibliography}
\end{document}